\documentclass{amsart}
\usepackage{mathrsfs}
\usepackage{amsmath,amsthm,amssymb}

\usepackage{stmaryrd}

\usepackage{booktabs}

\newtheorem{theorem}{Theorem}[section]
\newtheorem*{theorem2}{Theorem}
\newtheorem{prop}[theorem]{Proposition}
\newtheorem{lemma}[theorem]{Lemma}
\newtheorem{cor}[theorem]{Corollary}

\theoremstyle{definition}

\newtheorem{example}[theorem]{Example}

\newtheorem{question}[theorem]{Question}

\theoremstyle{remark}

\numberwithin{equation}{section}

\newcommand*{\eqnumref}[1]{%
  \begingroup
    \eqref{#1}%
  \endgroup
}

\begin{document}

\title{Wild Singularities of Kummer Varieties}


\author{Benedikt Schilson}
\address{}
\curraddr{}
\email{}
\thanks{}



\keywords{}

\date{}

\dedicatory{}

\begin{abstract}
In characteristic $p=2$, we compute the singularities of Kummer varieties arising from products of elliptic curves. This result is generalized to Kummer varieties associated to ordinary abelian varieties.
\end{abstract}

\maketitle

\tableofcontents

\section*{Introduction}
Let $k$ be an algebraically closed field of characteristic $p \geq 0$ and let $A$ be an abelian variety over $k$ of dimension $g \geq 2$. The Kummer variety $A/\iota$ is by definition the quotient of $A$ by the action of the sign involution $\iota$ and this quotient acquires singularities coming from the $2$-torsion of $A$. If $p \neq 2$, then the singular locus of the Kummer variety consists of $2^{2g}$ closed points.  In the case $p=2$ the situation changes dramatically. Here the number of singular points of $A/\iota$ can vary and is at most $2^g$. In the case of Kummer surfaces, Katsura \cite{katsura} studied the singularities and their resolution. For Kummer varieties of higher dimension, no such result is known. The goal of this paper is to determine the singularities for an accessible class of examples, namely Kummer varieties arising from products of elliptic curves.

To proceed so, one has to look for a suitable open affine $\iota$-invariant subset $U \subset A$ containing exactly one point of order (at most) $2$. Under this assumption, the quotient $U/\iota$ exists and is the desired open affine neighbourhood of the chosen singular point in $A/\iota$. By a change of variables, the spectrum of the affine coordinate ring  $R=\Gamma(U, \mathscr{O}_A)$ can be considered as a closed subscheme of affine space $\mathbb{A}^{2g}_k$ such that the $\mathbb{Z}/2\mathbb{Z}$-action on $R$ coincides with the induced action of an involution on $k[x_1, y_1, \ldots, x_g, y_g]$.  In the case that $A$ is the product of ordinary elliptic curves, the involution is given by $x_i \mapsto x_i$, $y_i \mapsto y_i + x_i$. The ring of invariants in this setting was computed by Richman \cite{richman}. Now the decisive step  is to prove that in this case passing to the quotient of $\mathbb{A}^{2g}_k$ by the group action is compatible with taking the closed subscheme mentioned above. For arbitrary group actions this statement does not hold.

From the description of the affine coordinate ring of $U/\iota$ as an affine $k$-algebra one immediately gets the singularity by completing the ring with respect to the ideal of the singular point. By an argument of Katsura, the computation of the singularities can be extended to abelian varieties whose associated formal group is isomorphic to that of a product of elliptic curves. This is the case e.g. for all ordinary abelian varieties and for all abelian varieties with $2^{g-1}$ points of order at most $2$.

For the sake of simplicity, only the case of ordinary abelian varieties is formulated here: The completed local ring at a singular point has a set of generators $T_{M}$, $X_i$ with $M \subset \lbrace 1, \ldots, g \rbrace$, $i = 1, \ldots, g,$ satisfying the relations
\begin{align*}
  &T_\emptyset=0, \quad \quad  X_i - T_{\lbrace i \rbrace}=0\, ,\quad \quad
  \sum_{L \subsetneq D} X_{D \smallsetminus L}\, T_L =0, \\
&T_A T_B + \sum_{L \subsetneq A \cap B} X_{(A \cap B) \smallsetminus L}\ P_L  \ T_{(A \cup B) \smallsetminus L} \\
&\quad \quad \quad \quad + P_{A \cap B}   \sum_{M \subsetneq A \smallsetminus B} X_{(A \smallsetminus B) \smallsetminus M} \ T_{M \cup (B \smallsetminus A)}=0
\end{align*}
for $1 \leq i \leq g$,  $A,B,D \subset \lbrace 1, \ldots, g \rbrace$ with $|D| \geq 3$ and $|A|, |B| \geq 2$. Here we use the notation
$$
X_M = \prod_{i \in M} X_i, \quad \quad P_M = \prod_{i \in M} (X_i^3+X_i).
$$

Now the  main result reads as follows:
\begin{theorem2}
Let $A$ be an ordinary abelian variety of dimension $g \geq 2$ over an algebraically closed field of characteristic $p=2$ and $X=A/\iota$ the Kummer variety of $A$. Then the completed local ring at every singular point of $X$  is isomorphic to
\begin{align*}
k\llbracket T_{M}, X_i \ | \  M \subset \lbrace 1, \ldots, g \rbrace, \ i = 1, \ldots, g \rrbracket /J,
\end{align*}
where the ideal $J$ is generated by all relations given above. The embedding dimension of this local ring equals $2^g - 1$.
\end{theorem2}
For Kummer surfaces, we get the known types of the singularities as described by Artin  \cite{artin_rdp}, Shioda  \cite{shioda} and Katsura \cite{katsura} in the 1970s.

The paper is structured as follows: First, we give normal forms for Weierstra\ss\ equations and find an open affine $\iota$-invariant subscheme $U=\textnormal{Spec}(R)$ of the product of elliptic curves. In Section~2 we get an explicit description of the quotient $U/\iota$ by taking the spectrum of the ring of invariants $R^\iota$.  Completing with respect to the singular point yields the singularity. This result is extended to arbitrary ordinary abelian varieties in Section~3. Finally, we  give a brief overview on related problems, namely quotient singularities arising from Artin--Schreier curves and on the rationality problem for Kummer varieties stemming from supersingular abelian varieties.


\medskip
\noindent
\textbf{Acknowledgement.} The author would like to thank Stefan Schr\"oer for many helpful discussions. This work is part of the author's PhD thesis and was conducted in the framework of the research training group
\emph{GRK 2240: Algebro-Geometric Methods in Algebra, Arithmetic and Topology},
which is funded by the DFG.

\section{Products of elliptic curves}
In this section $k$ denotes an algebraically closed ground field of characteristic $p=2$. The goal of this section is to find a suitable normal form for the action of the sign involution on an open subscheme of an elliptic curve.

Let $E$ be an ordinary elliptic curve with $j$-invariant $j_E \neq 0$ and pick an $\alpha \in k$ with $\alpha^6 = j_E^{-1}$. Then the Weierstra\ss\ equation
$
y^2 + \alpha x y = x^3 + \beta x
$
with coefficients $\alpha$ and $\beta=\alpha j_E^{-1}$ defines an elliptic curve with the same $j$-invariant. We may assume that $E$ is given by this equation. The spectrum of the ring
\[
R_{j_{E}}= k\left[x, y \right]/(y^2 + \alpha x y + x^3 + \beta x)
\]
yields an affine open set $U$, which contains all points of $E$ except the identity element. Hence, the sign involution $\iota$ maps $U$ to itself. By \cite{silverman}, Chapter~III, Algorithm~2.3, the action of $\iota$ is given by $(x,y)\mapsto (x, y +\alpha x)$; in particular, the point $(0,0) \in U$ is the unique point of order $2$.  The induced action on the affine coordinate ring $R_{j_E}$ is given by $$x \longmapsto x, \ \ \ \ \ y \longmapsto y + \alpha x.$$
It will turn out to be important that there is a linear action on $\mathbb{A}^2_k = \textnormal{Spec}(k[x,y])$ defined as above such that the action of $\iota$ on $\textnormal{Spec}(R_{j_E})$, considered as closed subscheme of $\mathbb{A}^2_k$, is induced by this linear action.

Now let $E$ be a supersingular elliptic curve, i.e. $j_E = 0$. The curve $E$ can be defined by the homogeneous Weierstra\ss\ equation $Y^2Z + YZ^2 = X^3$ and the sign involution $\iota$ operates on $E$ by $(X:Y:Z)  \mapsto (X:Y+Z:Z)$. The identity element $O=(0:1:0)$ is the only fixed point of $\iota$. Introducing new coordinates $x=X/Y$, $z=Z/Y$ yields the open affine subscheme
$
U'=\textnormal{Spec}\left(k\left[x,z \right]/(z+z^2+x^3)\right)
$
which contains all points of $E$ except $(0:0:1)$. An $\iota$-invariant open neighbourhood $U$ of $O$ is obtained by removing the point $\iota((0:0:1))=(0:1:1)$ from $U'$. The affine coordinate ring of $U$ is
\[
R_0=k \left[x, z, {(z+1)}^{-1} \right]/(z+z^2+x^3)
\]
and the induced automorphism of $k$-algebras $\iota^{*}\colon R_0 \rightarrow R_0$ is given by
\[
x \longmapsto \frac{x}{z+1}, \ \ \ \ \ z \longmapsto \frac{z}{z+1}.
\]
Moreover, we have $(z+1)^{-1} \mapsto z+1$.

The following lemma gives a simpler description of $R_0$ which requires only two generators.

\begin{lemma}
The $k$-algebra $R_0$ is generated by $x^2(z+1)^{-1}$ and $x(z+1)^{-1}$. More precisely: Let $S=k\left[ v, w \right]/(w^2+v^2w+v)$, then
\begin{align*}
&v \longmapsto \frac{x^2}{z+1}, \ \ \ \ \ w \longmapsto \frac{x}{z+1}.
\end{align*}
defines an isomorphism $\Phi\colon S \rightarrow R_0$.
\end{lemma}

\begin{proof}
The homomorphism $\Phi$ is well-defined: Of course there is a homomorphism $\tilde{\Phi} \colon k \left[ v, w \right] \rightarrow R_0$ which maps the indeterminates to the given elements, so it is enough to check that the given relation lies in the kernel of $\tilde{\Phi}$:
\begin{align*}
\tilde{\Phi}(w^2+v^2w+v) &= \frac{x^2}{(z+1)^2}+ \frac{x^4}{(z+1)^2} \cdot \frac{x}{z+1}+ \frac{x^2}{z+1} \\
 &= \frac{x^2}{(z+1)^3} \cdot \left(z+x^3+ z^2\right) = 0.
\end{align*}
Furthermore, simple calculations give
\begin{align*}
\Phi\left(vw+1\right)= \frac{1}{z+1}, \ \ \ \Phi\left(vw+v^3\right)=z, \ \ \ \Phi\left(w(vw+v^3+1)\right) = x,
\end{align*}
which proves surjectivity of $\Phi$. It remains to show that $\Phi$ is injective: The polynomial $w^2+v^2w+v \in k[v][w]=k[v,w]$ is irreducible (Eisenstein's criterion for prime element $v$), thus a prime element in the two-dimensional factorial ring $k[v,w]$. Consequently, the ring $S$ is a one-dimensional domain. Since both $R_0$ and  $S/\ker(\Phi)$ are domains, the kernel $\ker(\Phi) \subset S$ is a prime ideal. Hence, either $\ker(\Phi)=0$ holds or $\ker(\Phi)$ is a maximal ideal. The last case is impossible because $R_0$ is not a field.
\end{proof}

By abuse of notation, the induced involution $\Phi^{-1} \circ \iota^{*} \circ \Phi$ on $S$ will also be denoted by $\iota^{*}$.
\begin{lemma}\label{lemnonlin}
We have $\iota^{*}(v) = v$ and $\iota^{*}(w)=w+v^2$.
\end{lemma}
\begin{proof}
Straightforward computation: The element $x \cdot x(z+1)^{-1} \in R_0$ is  invariant as $\iota^{*}$ interchanges the factors. Next, we have
\begin{align*}
w + v^2 &= \Phi^{-1} \left( \frac{x}{z+1} + \frac{x^4}{z^2+1} \right) = \Phi^{-1} \left( \frac{x(z+1+x^3)}{z^2+1} \right) \notag \\ &= \Phi^{-1} \left( \frac{x(z^2+1)}{z^2+1} \right) = \Phi^{-1} \left( \iota^{*}\left( \frac{x}{z+1} \right) \right) = \Phi^{-1} \circ \iota^{*} \circ \Phi(w)=\iota^{*}(w).
\end{align*}
\end{proof}
The group action from Lemma~\ref{lemnonlin} is not induced by a linear action on $\mathbb{A}^2_k$, but it can still be regarded as closely related.


The $k$-algebras $R_{j}$ form the building blocks to build an open affine neighbourhood of one 2-division point. We can replace $R_0$ by the $k$-algebra $S$ and for $j_E\neq 0$ we can replace $R_{j_E}$ by the isomorphic ring $k\left[x, y \right]/(y^2 + x y + \alpha^{-3}x^3 + \alpha^{-1}\beta x)$.

The following proposition gives a summary of the situation established now:

\begin{prop}\label{situation}
Let $k$ be an algebraically closed field of characteristic $p=2$ and $A=E_1 \times \ldots \times E_g$ be a product of elliptic curves over $k$, where $E_1, \ldots, E_r$ are ordinary and $E_{r+1}, \ldots, E_g$ are supersingular. Then the spectrum of
\begin{align*}
R = k\left[ x_1, y_1, \ldots, x_r, y_r, \, v_{r+1}, w_{r+1}, \ldots, v_g, w_g \right] / I,
\end{align*}
with the ideal $I$ generated by
\begin{align*}
 &y_i^2 + x_i y_i + j_{E_i}^{\frac{1}{2}} x_i^3 + j_{E_i}^{-1} x_i,\ \  1 \leq i \leq r, \\
 &w_j^2+v_j^2w_j+v_j,\ \  r+1 \leq j \leq g,
\end{align*}
defines an open affine $\iota$-invariant neighbourhood of exactly one $2$-division point of $A$. The sign involution $\iota\colon A \rightarrow A$ induces a  $\mathbb{Z}/2\mathbb{Z}$-action on $R$ by
\begin{align*}
x_i &\longmapsto x_i, \ \ \ y_i \longmapsto y_i + x_i, \ \ \ \ \ 1 \leq i \leq r,  \\
v_j &\longmapsto v_j, \ \ \ w_j \longmapsto w_j + v_j^2, \ \ \ r+1 \leq j \leq g.
\end{align*}
\end{prop}

\section{Wild group actions}

Let $R$ be as in Proposition~\ref{situation} and $\hat{R}$ be the completion of $R$ at the origin. In the rings $k \llbracket x_i, y_i \rrbracket$ and $k\llbracket v_j, w_j \rrbracket$, the formal partial derivatives
\begin{align*}
 &\frac{\partial}{\partial x_i}(y_i^2 + x_i y_i + j_{E_i}^{\frac{1}{2}} x_i^3 + j_{E_i}^{-1} x_i)=y_i+ j_{E_i}^{\frac{1}{2}}  x_i^2+j_{E_i}^{-1}, \\
 &\frac{\partial}{\partial v_j}(w_j^2 + v_j^2 w_j + v_j)=  1
\end{align*}
are units, so there exist unique formal power series $x_i(y_i)$ and $v_j(w_j)$  satisfying the  respective relations in $\hat{R}$  (cf. \cite{bourbaki}, Chapter IV, page 37) and it follows that $\hat{R} \cong k \llbracket y_1, \ldots, y_r, w_{r+1}, \ldots, w_g  \rrbracket$. The induced group action of the sign involution on $\hat{R}$ is known to be non-linear, see e.g. \cite{peskin}, Proposition~2.1:

\begin{prop} 
Let $k$ be an algebraically closed field of characteristic $p>0$ and $S = k\llbracket u_1, \ldots , u_g\rrbracket$ equipped with a $G$-action, where the order of $G$ is divisible by $p$. If the morphism
$
\textnormal{Spec}(S) \rightarrow \textnormal{Spec}(S^G)$ is only at the maximal ideal $\mathfrak{m} \subset S$ ramified, then it is not possible to choose coordinates for $S$ such that $G$ acts linearly on $S$.
\end{prop}

However, the following statement shows that for specific diagonal actions one can obtain the invariants from a group action on affine space:

\begin{prop}\label{sr}
Let $k$ be a field of characteristic $p > 0$ and $G =\langle \sigma \rangle \cong \mathbb{Z}/p \mathbb{Z}$. Define a $G$-action on the polynomial ring $S= k \left[A_1, B_1, \ldots, A_g, B_g  \right]$ by
\begin{align}
 \sigma \left( A_i \right) &= A_i,  \quad \quad  \sigma \left( B_i \right) = B_i + A_i^{e_i} \notag 
\end{align}
for integers $e_i \geq 1$. Further, assume there are polynomials $f_i \in S$ of the form
\begin{equation} 
f_i = B_i^p - A_i^{(p-1)e_i} B_i + P_i(A_i) \notag 
\end{equation}
for some $P_i \in  A_i \, k[A_i]$, and denote by $I=(f_1, \ldots, f_g) \subset S$ the ideal generated by the $f_i$. Then the $G$-action on $S$ induces an action on the quotient $S/I$ and the residue class map  defines a surjective homomorphism $\phi\colon S^G \rightarrow \left( S/I \right)^G$ of $k$-algebras. In particular, \begin{equation}
S^{G}/I \cap S^{G} \cong \left( S/I \right)^G. \notag
\end{equation}
\end{prop}

\begin{proof}
As $\sigma(f_i )= f_i$ holds, it follows $\sigma(I) = I$ and one obtains the well-defined $G$-action $\sigma(h+I) =\sigma(h) + I$ on $S/I$.

Let $h+I \in S/I$ be an arbitrary element. There exists a representative $h_0 \in S$ with the property that the indeterminates $B_1, \ldots, B_g$ occur in every monomial of $h_0$  with exponent at most $p-1$, i.e.
$
\text{deg}_{B_i}(h_0) \leq p-1
$
for all $1 \leq i \leq g$. The polynomial
$$
\sigma(h_0)=h_0(A_1, B_1 + A_1^{e_1}, \ldots, A_g, B_g + A_g^{e_g})
$$
has the property $\text{deg}_{B_i}(\sigma(h_0)) \leq p-1$ as well; the same is true for $\sigma(h_0) - h_0 \in S$.

The element $h_0 + I \in S/I$ is invariant under $G$ iff $\sigma(h_0) - h_0 \in I$. In the case $\sigma(h_0) - h_0 =0$ one gets that $h+I = \phi(h_0)$ lies in the image of $\phi$. Hence, for surjectivity of $\phi$ it is enough to show that an element $h \in I$ which satisfies the condition $\text{deg}_{B_i}(h) \leq p-1$ for all $i$ has to be the zero polynomial.

Define the sets $M_i$ for $1 \leq i \leq g$ by
$$
M_i = \lbrace A_i^r, \, A_i^r B_i, \,\ldots, \,A_i^r B_i^{p-1},\, \left(B_i^p+ A_i^{(p-1)e_i} B_i + P_i \right)A_i^r B_i^s \ | \ r, s, \geq 0 \rbrace,
$$
so the elements of $M_i$ build up a basis of the $k$-vector space $k[A_i, B_i]$. Therefore, the elements of the set
$$
M= M_1 \cdot \ldots \cdot M_g = \lbrace m_1 \cdot \ldots \cdot m_g \ | \ m_i \in M_i, \ 1\leq i \leq g\rbrace
$$
are a basis of $S= k[A_1,B_1]\otimes \ldots \otimes k[A_g, B_g]$. The subspace $I$ has a generating set consisting of all elements of the form
$$
\left( B_i^{p}+ A_i^{(p-1)e_i} B_i + P_i \right) A_1^{r_1} B_1^{s_1} \cdot \ldots \cdot A_g^{r_g} B_g^{s_g}
$$
with $1 \leq i \leq g$ and exponents $r_1, s_1, \ldots, r_g, s_g \geq0$. As one can replace the factors $A_j^{r_j}B_j^{s_j}$ by a linear combination of elements of $M_j$, one gets a new generating set of $I$, consisting of all elements of the form
\begin{equation}
\left( B_i^p+ A_i^{(p-1)e_i} B_i + P_i \right)A_i^{r_i} B_i^{s_i} \cdot m_1^{(i)} \cdot \ldots \cdot \hat{m}_i \cdot \ldots \cdot m_g^{(i)} \notag 
\end{equation}
where $1\leq i \leq g$, $r_i, s_i \geq 0$, $m_j^{(i)}\in M_j$ for $1 \leq j \leq g$ and $\hat{m}_i$ is omitted in the product. This generating set contains exactly the elements of $M$ that have a factor $m_d$ with $\text{deg}_{B_d}(m_d) \geq p$.

Let now be $h \in I$ with $\text{deg}_{B_i}(h) \leq p-1$ for all $i$. Then on the one hand $h$ is a linear combination of elements of the generating set of $I$, on the other hand $h$ is linear combination of elements of $M$, namely of basis elements of the form $m_1 \cdot \ldots \cdot m_g$ with
$$
m_i \in \lbrace A_i^r,\ A_i^r B_i, \ldots, A_i^r B_i^{p-1} \ | \ r \geq 0 \rbrace.
$$
From the uniqueness of the linear combination it follows $h=0$. 
Hence, every element of $(S/I)^G$ is residue class of an invariant element of $S^G$.
\end{proof}
In particular, the proof shows that every (invariant) element $h+I \in S/I$ has a \emph{unique} (invariant) representative $h_0 \in S$ with $\textnormal{deg}_{B_i}(h_0) \leq p-1$ for every $1 \leq i \leq g$.

It is known that the isomorphism from the proposition always exists for finite group actions if the order of the group is coprime to the characteristic of the ground field. For wild group actions taking invariants and taking quotients may or may not commute. The easiest counterexample is the following:

\begin{example}
Let $S=\mathbb{F}_2[X,Y]$ with action of $G= \mathcal{S}_2$ by permuting the indeterminates and let $I$ be the $G$-invariant principal ideal $I=(X+Y)\subset S$. Then every element $f + I \in S/I$ is invariant. But the fundamental theorem of symmetric polynomials yields $S^G = \mathbb{F}_2[X+Y, \, XY]$, so the class  $X+I \in (S/I)^G$ does not lie in the image of the residue class map  $S^G \rightarrow (S/I)^G$. Otherwise there would exist polynomials $g(X,Y), h(X  Y) \in S$ such that $X = h(X  Y) + (X+Y) g(X,Y)$ in $S$, but coefficient comparison of the linear terms shows that this is impossible.
\end{example}

\section{Computation of invariants}
We start with the computation of the ring of invariants $R^\iota$ with $R$ and the group action as in Proposition~\ref{situation}. If all elliptic curves are ordinary, we can consider the action of the sign involution on $R$ as induced by a linear action on affine space $\mathbb{A}^{2g}_k$.  As soon as the ring of invariants in this situation is known, the ring $R^\iota$ is then obtained by using Proposition~\ref{sr}.

The involution on $k[x_1, y_1, \ldots, x_g, y_g]$ given by $x \mapsto x$, $y \mapsto y+x$ has been studied by Richman \cite{richman}, who computed a set of generators for the ring of invariants over a field of characteristic $p=2$. Here we give the more general result by Campbell and Hughes (cf. \cite{campbell}, page 4), which applies for the case $p>0$.

\begin{prop}[Campbell, Hughes] \label{campbell_erz}
Let $K= \mathbb{F}_p$ and $G=\left<{\sigma} \right>$ cyclic of order $p$ and $S_K=K\left[x_1, y_1, \ldots, x_g, y_g \right]$  with  $G$-action given by $\sigma(x_i)=x_i$, $\sigma(y_i)=y_i + x_i$ for $1\leq i\leq g$. Then the $G$-invariant elements
\begin{itemize}\itemsep2pt
 \item[$ $] $x_i \ \textnormal{ for } 1 \leq i \leq g, $
 \item[$ $] $x_i y_j - x_j y_i \  \textnormal{ for } 1\leq i < j \leq g,$
 \item[$ $] $\textnormal{N}(y_i) \ \textnormal{ for } 1 \leq i \leq g,$
 \item[$ $] $\textnormal{Tr}\left( y_1^{a_1}\cdots y_g^{a_g} \right) \textnormal{ for } 0 \leq a_1, \ldots, a_g \leq p-1 \textnormal{ and } \sum_i a_i > 2(p-1)$
\end{itemize}
build up a generating system of the invariant ring $S_{K}^{G}$. Here $\textnormal{N}(X)= \prod_{\sigma \in G} \sigma (X)$ denotes the norm and $\textnormal{Tr}(X)=\sum_{\sigma \in G} \sigma (X)$ the trace of an element $X \in S_{K}$.
\end{prop}

As the equality $\textnormal{N}(y_i)= y_i^2+x_i y_i = j_{E_i}^{\frac{1}{2}} x_i^3 + j_{E_i}^{-1} x_i$ holds in $R$, one can omit these elements from a generating system of $R^\iota$. Furthermore, $x_i = \textnormal{Tr}(y_i)$ and $x_i y_j - x_j y_i = \textnormal{Tr}(y_i y_j) + x_1 x_2$ can be written as traces, so we get the following corollary:

\begin{cor}\label{erzeuger}
Assume that $A$ is the product of $g$ ordinary elliptic curves. Then the ring $R^\iota$ (with $R$ and group action as in Proposition~\ref{situation}) is generated by all elements of the form $\textnormal{Tr}(y_{i_1} \cdot \ldots \cdot y_{i_m})$.
\end{cor}

The ring $R^\iota$ can also be computed directly without using Richman's result or Proposition~\ref{campbell_erz}. Observe that
\begin{align*}
f_M := \textnormal{Tr}\left( \prod_{i \in M} y_i \right) = \sum_{d=0}^{|M|-1} \ \sum_{\substack{N \subset M, \\ |N| = d}} \ \ \prod_{m \in M \smallsetminus N} x_m \ \prod_{n \in N} y_n
\end{align*}
holds, in particular we have $\textnormal{deg}_{x_i}(Q) + \textnormal{deg}_{y_i}(Q) = 1$ for every monomial $Q$ of $f_M$ and $f_M$ is the sum of all such monomials except of $\prod_{i \in M} y_i$. We regard the $f_M$ as elements of the vector space
$$
V= \lbrace f \in k(x_1, \ldots, x_g)[y_1, \ldots, y_g]  \ | \ \textnormal{deg}_{y_i}(f) \leq 1 \textnormal{ for every } 1 \leq i \leq g \rbrace
$$
over the field $k(x_1, \ldots, x_g)$ and compute the eigenvectors for the eigenvalue $\lambda = 1$ of the linear map induced by $\iota$. If $a_1, \ldots, a_l \in V$ is a basis of the eigenspace $V^\iota$, then the elements $a_1, \ldots, a_l, y_g a_1, \ldots, y_g a_l $ are linearly independent in $V$. In fact, if $0= \mu_1 a_1 + \ldots + \mu_l a_l + \lambda_1 y_g a_1 + \ldots + \lambda_l y_g a_l$ is a linear combination of zero, then applying the linear map  $\iota -1$ yields $0 = x_g (\lambda_1 a_1 + \ldots + \lambda_l a_l)$ which is only possible if the coefficients $\lambda_i$ vanish for every $i$; as a consequence, also all $\mu_i$ vanish. Hence $l \leq 2^{g-1}$. On the other hand, a similar argument shows that the $2^{g-1}$ elements of the form $f_{M \cup \{g\}}$ with $M \subset \{ 1, \ldots, g-1 \}$ are linearly independent, thus $l = 2^{g-1}$ and a basis of $V^\iota$ is found. In the last step one has to compute the set $V^\iota \cap k[x_1, y_1, \ldots, x_g, y_g]$. If $f= \sum_M {\lambda_{M \cup \{g \}}} \cdot f_{M \cup \{g\}}$ is a linear combination with coefficients from $k(x_1, \ldots, x_g)$, then one can reduce the denominator of $\lambda_{M \cup \{g \}}$ either to $1$ or $x_g$: Choose a subset $M'$ of maximal cardinality such that $\lambda_{M' \cup \{g \}}$ is non-zero. Then $f_{M' \cup \{ g \}}$ is the only element in the sum that contains the monomial $x_g \cdot \prod_{i \in M'} y_i$, hence the denominator of $\lambda_{M' \cup \{ g \}}$ has to divide $x_g$. Apply the same argument to $f + \lambda_{M' \cup \{ g \}} \cdot f_{M' \cup \{g\}}$ to get the statement for all denominators.

Now it suffices to consider
$$
f= \sum_{M \subset \{ 1, \ldots, g-1 \}} \frac{\mu_{M \cup \{g\}}}{x_g} \cdot f_{M \cup \{g\}} \in k[x_1, y_1, \ldots, x_g, y_g]
$$
with $\mu_{M \cup \{ g\}} \in k[x_1, \ldots, x_{g-1}]$. The formal partial derivative $\partial f/\partial y_g$ vanishes as $\textnormal{deg}_{y_g}(f)\leq 0$. On the other hand we have $\textnormal{deg}_{x_g}(f)=0$, so $f$ is an invariant element of  $k[x_1, y_1, \ldots, x_{g-1}, y_{g-1}]$. By induction the statement follows. (The case $g=1$ is trivial: Here $R^\iota = k[x_1]$, generated by $x_1 = \textnormal{Tr}(y_1)$.)

Next, we consider the case that supersingular factors appear in the product of elliptic curves. Nearly the same computation is possible:

\begin{prop}\label{invgen2}
With the notation as in Proposition~{\ref{situation}}, the invariant ring $R^\iota$ is generated by the elements
\begin{align*}
 &v_{r+1}, \ldots, v_g , \notag \\
 &f_{M \cup N}=\textnormal{Tr}\left( \prod_{i \in M} y_i \cdot \prod_{j \in N} w_j \right),
\end{align*}
where $M \subset \lbrace 1, \ldots, r \rbrace$, $N \subset \lbrace r+1, \ldots, g  \rbrace$ run through all subsets.
\end{prop}

\begin{proof}
Consider  the $k$-algebras $S_{n} = k\left[x_1, y_1, \ldots, x_r, y_r,\, v_{r+1}^n, w_{r+1}, \ldots, v_g^n, w_g \right]$   for $n \in \{ 1, 2\}$, 
and the vector spaces
\begin{align*}
V_{n} &=\lbrace f \in S_{n} \ |\ \textnormal{deg}_{y_i}(f) \leq 1,\ \textnormal{deg}_{w_j}(f) \leq 1\ \textnormal{for every}  \ 1\leq i \leq r < j \leq g  \rbrace \subset S_n
\end{align*}
with the $\mathbb{Z}/2\mathbb{Z}$-action given by $$\iota(x_i) = x_i, \quad \iota(y_i) = y_i + x_i, \quad {\iota}(v_j^n)=v_j^n, \quad {\iota}(w_j)=w_j + v_j^2.$$
For every element of $R=S_1/I$ we can find a unique representative in $V_1$. Let $f_1 \in V_1$ be a polynomial and let $L \subset \lbrace r+1, \ldots, g \rbrace$  be a subset. Let $h_L \in V_1$ be the sum of all monomials $Q$ of $f_1$, for which the condition
\begin{equation} \notag
j \in L \, \Longleftrightarrow \, \textnormal{deg}_{v_j}(Q) \equiv 1 \ \textnormal{mod } 2
\end{equation}
holds. The polynomial $h_L$ is the sum of all monomials of $f_1$, where the indeterminates $v_j$ with $j \in L$ occur with odd exponents, whereas the other $v_i$ have even exponents. Of course, $f_1$ is the sum of the $h_L$ and we get
\begin{equation} \notag
f_1 = \sum_L h_L = \sum_{L} \left( \prod_{j \in L} v_j \right) \tilde{h}_L
\end{equation}
for suitable polynomials $\tilde{h}_L \in V_2$. As ${\iota}(V_2) = V_2$ and ${\iota}(v_j)=v_j$, an element $f_1 \in V_1$ is $\iota$-invariant iff the polynomials $\tilde{h}_L \in V_2$ are $\iota$-invariant.
The action of $\iota$ on $V_2$ is of the form for which the invariant elements are known. By Corollary~\ref{erzeuger}, the residue class of every invariant element $f_2 \in V_2^{{\iota}}$ can be written as polynomial in the traces, hence the $f_{M \cup N}$ and $v_j$ build up a generating set of the $k$-algebra $R^\iota$.
\end{proof}

In the setting of Proposition~\ref{campbell_erz}, the relations between the generators are unknown in general. However, in the case $p=2$ (which we are interested in) Campbell and Wehlau gave a generating set for the ideal of relations in \cite{campbell2}, Theorem~3.6. These relations are

\begin{align*}
(i)\quad &\sum_{L \subsetneq A} x_{A \smallsetminus L} \, \textnormal{Tr}\left(y_{L}\right) = 0,  \\
(ii) \quad &\textnormal{Tr}\left(y_A\right) \, \textnormal{Tr}\left(y_B\right) = \sum_{L \subsetneq A \cap B} x_{(A \cap B) \smallsetminus L} \ \textnormal{N}\left(y_L\right) \, \textnormal{Tr}\left(y_{(A \cup B) \smallsetminus L}\right) \quad \quad \quad \quad \quad \quad \quad \notag \\
 &\quad \quad \quad \quad \quad \quad \quad \quad \quad \quad \quad +\textnormal{N}\left(y_{A \cap B}\right)  \sum_{M \subsetneq A \smallsetminus B} x_{(A \smallsetminus B) \smallsetminus M} \ \textnormal{Tr}\left(y_{M \cup (B \smallsetminus A)} \right),
\end{align*}
where we use the compact notation $x_M = \prod_{i \in M} x_i$, same for $y_M$. For every subset $M \subset \lbrace 1, \ldots, g \rbrace$ let $M' = M \cap \lbrace 1, \ldots, r \rbrace$, $M'' = M \cap \lbrace r+1,  \ldots, g \rbrace$. In the factor ring $R$ we can simplify the formulas by replacing the norm by suitable polynomials in the indeterminates $x_i$ and $v_j$:
\begin{align*}
 &\textnormal{N}\left(y_{L'} w_{L''}\right) = \textnormal{N}\left(y\right)_{L'} \cdot \ \textnormal{N}\left( w\right)_{L''} =P_{L'} \cdot \ P_{L''} = P_L, \\
 &\textnormal{N}\left(y_{A' \cap B'}\ w_{A'' \cap B''}\right) = P_{{A' \cap B'}} \cdot \ P_{A'' \cap B''} = P_{A \cap B}
\end{align*}
and the polynomial $P_i$ is given by $P_i = \gamma_i x_i^3 + \gamma_i^{-2} x_i$ for $i \in M'$ where $\gamma_i^2$ equals the $j$-invariant of the elliptic curve $E_i$, and $P_j = v_j$ for $j \in M''$. Again, we write $P_L= \prod_{i \in L} P_i$.
Now for arbitrary subsets $A, B, D \subset \lbrace 1, \ldots, g \rbrace$, the following holds in $R^\iota$:
\begin{align}
(i)\ &\sum_{L\subsetneq D} x_{D' \smallsetminus L'} \, v^2_{D'' \smallsetminus L''} \, \textnormal{Tr}\left(y_{L'} w_{L''}\right) = 0,\notag \\
(ii) \ &\textnormal{Tr}\left(y_{A'} w_{A''}\right) \, \textnormal{Tr}\left(y_{B'} w_{B''}\right)\notag \\ &\ =\sum_{{L \subsetneq A \cap B }} x_{(A' \cap B') \smallsetminus L'} \ v^2_{(A'' \cap B'') \smallsetminus L''}\ P_L \, \textnormal{Tr}\left(y_{(A' \cup B') \smallsetminus L'} \ w_{(A'' \cup B'') \smallsetminus L''}\right)  \notag \\
 &\  +P_{A \cap B}  \sum_{{M \subsetneq A \smallsetminus B}} x_{(A' \smallsetminus B') \smallsetminus M'} \ v^2_{(A'' \smallsetminus B'') \smallsetminus M''} \ \textnormal{Tr}\left(y_{M' \cup (B' \smallsetminus A')}\ w_{M'' \cup (B'' \smallsetminus A'')} \right) \notag
\end{align}

We now show that these relations again generate all relations between the generators in $R^\iota$.

\begin{prop}\label{endlerzkalg}
Let $R$ and $R^\iota$ be as in Proposition~\ref{situation}. Furthermore, we define
\begin{align*}
R_0=k\left[T_{M}, X_i, V_j \ | \ M \subset \lbrace 1, \ldots, g \rbrace, \ i = 1, \ldots, r,\ j=r+1, \ldots, g  \right]
\end{align*}
and the homomorphism of $k$-algebras $\psi\colon R_0 \rightarrow R^\iota$ by
$T_{M} \mapsto \textnormal{Tr}(y_{M'} \, w_{M''})$, $X_i \mapsto x_i$,  $V_j \mapsto v_j$.
Then $\psi$ is surjective and $\textnormal{ker}(\psi)$ is generated by
\begin{align}
   &T_\emptyset\, , \quad X_i - T_{\lbrace i \rbrace}\, ,  \quad V_j^2 - T_{\lbrace j \rbrace}\, , \label{Typnot2} \\
  &\sum_{{L \subsetneq D}} X_{D' \smallsetminus L'}\, V_{D'' \smallsetminus L''}^2 \, T_{L} \label{Typleicht2}, \\
&T_A T_B + \sum_{L \subsetneq A \cap B} X_{(A' \cap B') \smallsetminus L'} \, V^2_{(A'' \cap B'') \smallsetminus L''}\ P_L  \ T_{(A \cup B) \smallsetminus L} \label{Typschwer2} \\
 &\quad \quad \quad + P_{A \cap B}   \sum_{M \subsetneq A \smallsetminus B} X_{(A' \smallsetminus B') \smallsetminus M'}\, V^2_{(A'' \smallsetminus B'') \smallsetminus M''} \ T_{M \cup (B \smallsetminus A)}  \notag
\end{align}
for $1 \leq i \leq r < j \leq g$, $A, B, D \subset \lbrace 1, \ldots, g \rbrace$ with $|A|, |B| \geq 2$, $|D| \geq 3$.
\end{prop}
Note that the relations of type \eqnumref{Typleicht2} and \eqnumref{Typschwer2} are trivial if  $|D| \leq 2$ or $|A| \leq 1 $ or $|B| \leq 1$ holds. The relations in \eqnumref{Typnot2} are only used to get a better notation.
\begin{proof}
The homomorphism $\psi$ is surjective, as a generating set of $R^\iota$ lies in the image. Let $J$ be the ideal generated by all elements of the form \eqnumref{Typnot2}-\eqnumref{Typschwer2}. The relations from $J$ are contained in the kernel of $\psi$. Set $\bar{R}_0= R_0/J$. We show that $\psi$ induces an isomorphism $\bar{\psi}\colon \bar{R}_0 \rightarrow R^\iota$.

Take an element $\bar{Q} \in \textnormal{ker}(\bar{\psi})$, i.e. $\bar{Q}$ is the residue class of a polynomial in ${R}_0$ that vanishes in $R^\iota$ when plugging in the generators. Using the relations of type \eqnumref{Typschwer2}, we can find a representative $Q_0$ of $\bar{Q}$ in the polynomial ring $R_0$ such that the condition $\textnormal{deg}_{T_M}(Q_0)\leq 1$ holds for all indeterminates $T_M$ with $|M| \geq 2$.  This means that $Q_0$ gives a relation in $R^\iota$ of the form
$$
 \sum_{|M| \geq 2} \ \mu_M \, \textnormal{Tr}(y_{M'} w_{M''}) \ = \ 0
$$
for suitable polynomials $\mu_M \in k[x_1, \ldots, x_r, v_{r+1}, \ldots v_g]$. Of course, this relation also holds in $R$. The proof of Proposition~\ref{sr} shows that it holds in the polynomial ring $S=k[x_1, y_1, \ldots, v_g, w_g ]$ as well and in $S^\iota$, as all occuring polynomials are invariant. Thus, $Q_0$ is a relation between the generators of the invariant ring $S^\iota$, hence generated by the relations found by Campbell and Wehlau. The ideal $J$ is generated by the residue classes of these relations and, as a consequence, $\bar{Q}=0$ in $\bar{R}_0$. So $\bar{\psi}$ is an isomorphism.
\end{proof}

Now one can read off the embedding dimension of the singularity. Recall that the embedding dimension of a noetherian local ring is by definition the cardinality of a minimal generating set of its maximal ideal. The embedding dimension of the singularity is the embedding dimension of the completion $\hat{R}^\iota$ of $R^\iota$ with respect to the singular point.

\begin{prop}\label{miniezs}
The invariant elements
\begin{equation}
 x_1, \ldots, x_r, \, v_{r+1}, \ldots, v_g , \ f_{M}=\textnormal{Tr}(y_{M'}\, w_{M''})\  \textnormal{ with }  |M| \geq 2\notag
\end{equation}
form a minimal set of generators of $R^\iota$. Furthermore, the embedding dimension of $\hat{R}^\iota$ is $2^{g}-1$.
\end{prop}

\begin{proof}
Let $\mathfrak{m} \subset \hat{R}^\iota$ denote the maximal ideal, which is generated by the images of the generators of $R^\iota$ in the completion. The minimality follows from the statement about the embedding dimension: Suppose there is a generating set $a_1, \ldots, a_d$ of $R^\iota$ with $d < 2^{g}-1$. These elements are residue classes of polynomials in $x_i, y_i, v_j, w_j$ and without loss of generality we can assume that their constant terms are zero. Hence, the ideal $(a_1, \ldots, a_d) \subset R^\iota$ is the maximal ideal of the singular point in  $\textnormal{Spec}(R^\iota)$ and we get a set of generators of $\mathfrak{m}$ consisting of less than $2^{g}-1$ elements, which is a contradiction.

We compute the embedding dimension by showing that the elements $x_i, v_j, f_M$ are linearly independent in the cotangent space $\mathfrak{m}/\mathfrak{m}^2$. This vector space has dimension at most $2^g -1$ because the residue classes of a minimal generating set of $\mathfrak{m}$ form a basis of the vector space.

After replacing $\hat{R^\iota}$ by the completion of the ring ${R_0}/\textnormal{ker}(\psi)$ from Proposition~\ref{endlerzkalg},  consider the linear combination 
\begin{align} \label{linearkombination}
\sum_{i=1}^r \lambda_i \, X_i \ + \ \sum_{i=r+1}^g \lambda_i \, V_i \ + \sum_{|M| \geq 2} \mu_M \, T_M \in \mathfrak{m}^2
\end{align}
in $\hat{R}^\iota$ with coefficients $\lambda_i, \mu_M \in k$. Denote the element \eqnumref{linearkombination} by $f$. The ideal  $\mathfrak{m}^2 \subset \hat{R^\iota}$ is generated by all products of two elements of the known generating set. The relations of type  \eqnumref{Typschwer2} show that every such product is contained in $\mathfrak{a}=(X_1, \ldots, X_r, V_{r+1}, \ldots V_g) \subset \hat{R^\iota}$; hence, $f \in \mathfrak{a}$. In $\hat{R^\iota}/\mathfrak{a}$, the condition on $f$ reads
\begin{align*}
  &\sum_{|M| \geq 2} \mu_M \, T_M \in \left(T_A \cdot T_B, \, |A|, |B| \geq 2\right) \subset k\llbracket \,T_M\,|\ |M| \geq 2 \rrbracket,
\end{align*}
so $\mu_M=0$ for every $M$. Similarly, one gets $\lambda_i=0$ by reducing \eqnumref{linearkombination} modulo  $(T_M, \, |M|\geq 2) + \mathfrak{m}^2 \subset \hat{R^\iota}$.
\end{proof}

For abelian surfaces (i.e. $g=2$), we recover some known results:
\begin{example}\label{bspfl}
Let $A=E_1 \times E_2$ be the product of two elliptic curves and let $0 \leq r \leq 2$ be the number of supersingular factors in the product. In this case a neighbourhood of a singular point of $A/\iota$ is given by the spectrum of
\begin{align*}
 &R_2=k[ X_1,X_2,T]/\left(T^2 + X_1 X_2 T + X_1^2 P_2(X_2) + X_2^2 P_1(X_1)\right), &\textnormal{if } r=2,     \\
 &R_1=k[ X_1, V_2,T]/\left( T^2 + X_1 V_2^2 T + X_1^2 V_2 + V_2^4 P_1(X_1) \right), &\textnormal{if } r=1,     \\
 &R_0= k[ V_1,V_2,T]/\left(T^2 + V_1^2 V_2^2 T + V_1^4 V_2 + V_2^4 V_1 \right), &\textnormal{if } r=0.
\end{align*}
Recall that $P_i(X_i)=X_i \cdot U_i(X_i)$ where $U_i$ becomes a unit in the completion because of its non-zero constant term. Hence, by a change of variables  $T'=T U_1^{-1} U_2^{-1}$, $X_1' = X_1 U_1^{-1}$, $X_2' = X_2 U_2^{-1}$ in $R_2$, we get 
the formal completion
\begin{align*}
\hat{R_2}= k \llbracket X_1',X_2',T' \rrbracket/\left(T'^2 + X_1' X_2' T' +  X_1'^2 X_2'+  X_1' X_2'^2 \right)
\end{align*}
which is a singularity of type $D_4^1$ in \cite{artin_rdp}. Similarly, we obtain
\begin{align*}
\hat{R_1}= k \llbracket X_1', V_2, T' \rrbracket / \left( T'^2 + X_1' V_2^2 T' + X_1'^2 V_2 + V_2^4 X_1' \right)
\end{align*}
which is a singularity of type $D_8^2$, in accordance with \cite{schroeerhilb}, Propositions~5.1 and 5.2. In the case $r=0$ the equation for the singularity is already the normal form from \cite{katsura}, Proposition~8.
\end{example}

In the case of Kummer threefolds arising from products of elliptic curves, one gets seven generators for the invariant ring $R^\iota$ and the ideal of relations is generated by ten relations of type \eqnumref{Typschwer2} and one relation of type \eqnumref{Typleicht2}.

\section{Formal groups}
If two abelian varieties have isomorphic associated formal groups, then the Kummer varieties of these abelian varieties have "formal isomorphic" singularities (the completed local rings at the singular points are isomorphic). This argument is used by Katsura \cite{katsura}, Proposition~3, where the case of ordinary abelian surfaces is settled by looking at the product of elliptic curves. In this section, we sum up some results on formal groups in order to use Katsura's argument in the higher dimensional case.

Here $k$ denotes a ground field of characteristic $p>0$. When talking of a formal group, we will usually mean a formal spectrum $G=\textnormal{Spf}(\hat{\mathcal{O}})$ with $\hat{\mathcal{O}}$ a local noetherian $k$-algebra, endowed with morphisms  that define the group structure. The formal group associated to an algebraic group is the formal completion at the identity element. An isogeny between formal groups is a homomorphism  that becomes an isomorphism in the factor category of commutative formal groups over $k$ modulo the full subcategory of formal spectra $\textnormal{Spf}(\Lambda)$ with $\Lambda$ artinian.

Manin \cite{manin} gave a classification of commutative formal groups up to isogeny:
Every finite-dimensional commutative formal group  $G$ over  $k$ is isogenous to a sum  $G \sim T \oplus U \oplus V$ where
\begin{align*}
 T = \bigoplus_{i=1}^r G_{1,0} \ , \quad U = \bigoplus_{n\geq1}  G_{n, \infty}^{\oplus r_n}\ , \quad V = \bigoplus_{\substack{n, m \geq 1 \\ (n,m)=1, }} G_{n,m}^{\oplus s_{n,m}} \
\end{align*}
for natural numbers $r$, $r_n$, $s_{n,m}$. This decomposition is unique up to isogeny.
The formal groups $G_{n,m}$,   $1 \leq n < \infty$, $0 \leq m \leq \infty$, $\textnormal{gcd}(n,m)=1$ for $m \neq \infty$ can be characterized (up to isogeny)  by the following properties:
\begin{itemize}
\item[$(i)$] $\dim(G_{n,m})=n$.
\item[$(ii)$] $G_{n,m}$ is indecomposable.
\item[$(iii)$] For $G_{n,m}$ multiplication by $p$ is an isogeny of degree $p^{n+m}$ ($m \neq \infty$).
\end{itemize}
In particular, $G_{n,m}$ and $G_{n',m'}$ with $(n,m)\neq(n',m')$ lie in  different isogeny classes. If $k$ is algebraically closed and $G$ is reduced, one can replace $G \sim T$  by $G \cong T$ (cf.~\cite{manin}, Theorem~1.2 on page 20).

Lemma 1 in \cite{manin2} shows how to compute the degree of induced isogenies between formal groups: If $\alpha \colon A \rightarrow B$ is an isogeny between algebraic groups and $\hat{\alpha}\colon \hat{A} \rightarrow \hat{B}$ denotes the isogeny between the formal groups, then $
 \textnormal{deg}\left(\hat{\alpha}\right) = \textnormal{insdeg}\left( \alpha \right)$. In the case $\alpha=p_A\colon A \rightarrow A$ for an abelian variety $A$ of dimension $g$, one gets
\begin{align}
\textnormal{insdeg}(p_A)=\frac{\textnormal{deg}(p_A)}{\textnormal{sepdeg}(p_A)}=\frac{p^{2g}}{|\textnormal{ker}(p_A)(k)|} \notag 
\end{align}
which can be used to determine the formal group of an elliptic curve $E$. If $E$ is ordinary, then  $\hat{E}\cong G_{1,0}$, and for supersingular elliptic curves $\hat{E} \cong G_{1,1}$ holds. As these formal groups are one-dimensional, "isogenous" can be replaced by "isomorphic" (cf. \cite{hazewinkel}, Theorem~18.5.1).

Formal groups arising from abelian varieties are far more special. Here a certain kind of symmetry condition holds which is due to Manin for finite fields (cf. \cite{manin}, Theorem~4.1) and Oort for algebraically closed fields (cf. \cite{oort}, page III.19-3): The formal group of an abelian variety $A$ can be written as a sum of the form
\begin{align}
\hat{A} \sim G_{1,0}^{\oplus r} \oplus G_{1,1}^{\oplus s} \oplus \bigoplus_{\substack{m > n \geq 1 \\ (n,m)=1 }} \left( G_{n,m} \oplus G_{m,n}  \right)^{\oplus t_{n,m}}  \notag 
\end{align}
for suitable $r, s, t_{n,m} \geq 0$, i.e. the summands $G_{n,m}$ and $G_{m,n}$ occur with the same multiplicity and $G_{n, \infty}$ does not appear.

\begin{table}
\centering
\begin{small}
\begin{tabular}{@{}cl@{}}\toprule
$\textnormal{dim}(G)$  & possible $G$ up to isogeny \\ \midrule
\parbox[0pt][1.5em][c]{0cm}{} 1 & $G_{1,0}, \ \ G_{1,1}$      \\
\parbox[0pt][1.6em][c]{0cm}{} 2 & $G_{1,0}^{\oplus 2}, \ \ G_{1,1}^{\oplus 2}, \ \ G_{1,0} \oplus G_{1,1}$    \\
\parbox[0pt][1.6em][c]{0cm}{} 3   & $G_{1,0}^{\oplus r} \oplus G_{1,1}^{\oplus 3-r}, \, 0 \leq r \leq 3, \ \, \ G_{1,2} \oplus G_{2,1}$  \\
\parbox[0pt][1.6em][c]{0cm}{} 4   & $G_{1,0}^{\oplus r} \oplus G_{1,1}^{\oplus 4-r}, \, 0 \leq r \leq 4, \ \, \  G_{1,i}\oplus G_{1,2} \oplus G_{2,1}, \ i \in \{0, 1\}, \ \ G_{1,3} \oplus G_{3,1}$  \\ \bottomrule
\end{tabular}
\end{small}
\vspace{1mm}
\caption{Formal groups of abelian varieties in small dimension}\label{formalklein}
\end{table}

\begin{prop}
Let $A$ be a $g$-dimensional abelian variety over $k$ with $p=2$.
\begin{itemize}
\item[$(i)$] If $A$ is ordinary, i.e. $A$ has $2^g$ points of order at most $2$, then $\hat{A} \cong G_{1,0}^{\oplus g}$.
\item[$(ii)$] If $A$ has $2^{g-1}$ points of order at most $2$, then  $\hat{A} \cong G_{1,1} \oplus G_{1,0}^{\oplus g-1}$.
\end{itemize}
\end{prop}

\begin{proof}
In both cases the group scheme $\textnormal{ker}(2_A)$ can be described explicitly as
\begin{align*}
(i) \quad (\mathbb{Z}/2 \mathbb{Z})^{\oplus g} \oplus \mu_{2}^{\oplus g}, \quad \quad   (ii) \quad (\mathbb{Z}/2 \mathbb{Z})^{\oplus g-1} \oplus \mu_{2}^{\oplus g-1} \oplus N_{2},
\end{align*}
where $N_{2}$ is a local-local group scheme. The formal group associated to $A$ now arises as limit over the local part (cf. \cite{tate}, Examples on page 166)
\begin{align*}
(i) \quad \hat{A}  \ &\cong \ \varinjlim \mu_{2^n}^{\oplus g} \ \cong \ (\varinjlim \mu_{2^n})^{\oplus g} \ \cong  \ \hat{\mathbb{G}}_m^{\oplus g}\ \cong \ G_{1,0}^{\oplus g},  \\
(ii) \quad \hat{A} \ &\cong  \ \varinjlim \left( \mu_{2^n}^{\oplus g-1} \oplus N_{2^n} \right) \ \cong \ (\varinjlim \mu_{2^n})^{\oplus g-1} \oplus \varinjlim N_{2^n} \\ &\cong \ G_{1,0}^{\oplus g-1} \oplus H.
\end{align*}
Here $H$ denotes a  one-dimensional commutative formal group. From the uniqueness of the decomposition of $\hat{A}$ up to isogeny it follows that $H$ is isogenous to $G_{1,1}$. As $H$ is one-dimensional, one can replace "isogenous" by "isomorphic".
\end{proof}

In the case $A \sim G_{1,0}^{\oplus g-2} \oplus G_{1,1}^{\oplus 2}$ the above argument does not work because there are several isomorphism classes within the isogeny class of $G_{1,1}^{\oplus 2}$. (As a consequence, there are two different types of singularities of wild Kummer surfaces with one singular point in \cite{katsura}.) The relation between the formal group of $A$ and the singularities of $A/\iota$ is as follows:

\begin{prop}\label{fgq}
Let $A$ and $B$ be abelian varieties over an algebraically closed field $k$ of characteristic $p=2$.  If the formal groups $\hat{A}$ and $\hat{B}$ are isomorphic, then the singularities of the Kummer varieties $X=A/\iota_A$ and $Y=B/\iota_B$ are formal isomorphic (meaning that the completed local rings are isomorphic).
\end{prop}

\begin{proof}
We denote by $\phi\colon \hat{A} \rightarrow \hat{B}$ the isomorphism. The schemes $A$ and $B$ have the same dimension, hence $\hat{A} \cong \hat{B} \cong \textnormal{Spf}(R)$ as formal spectra with  $R= k\llbracket T_1, \ldots, T_g \rrbracket$.
We obtain automorphisms  $\phi^{*}$, $\iota_A^{*}$, $\iota_B^{*}$ of   $R$ induced by $\phi$ and the sign involutions of $A$ and $B$. Moreover, the  automorphisms $\iota_A^{*}$, $\iota_B^{*}$ coincide with the  induced automorphisms of the completed local rings $\hat{\mathscr{O}}_{A,e_A}$ and $\hat{\mathscr{O}}_{B,e_B}$.

It suffices to show that the invariant rings $R^{\iota_A^{*}}$ and $R^{\iota_B^{*}}$ are isomorphic, because  taking invariants "commutes" with completion. From the compatibility with the group laws one immediately gets $\iota_A^{*}= \phi^{*} \circ \iota_B^{*} \circ (\phi^{*})^{-1}$.  We now show  $R^{\iota_A^{*}}=\phi^{*}(R^{\iota_B^{*}})$: If $f \in R^{\iota_B^{*}}$ is an arbitrary element, then
\begin{align*}
\iota_A^{*}(\phi^{*}(f)) = \phi^{*} \circ \iota_B^{*} \circ (\phi^{*})^{-1} \circ \phi^{*}(f) = \phi^{*}(f),
\end{align*}
so $\phi^{*}(f) \in R^{\iota_A^{*}}$. Switching the roles of $\phi^{*}$ and ${\phi^{*}}^{-1}$, one gets  the inclusion ${\phi^{*}}^{-1}(R^{\iota_A^{*}}) \subset R^{\iota_B^{*}}$ with the same argument. Apply $\phi^{*}$ to complete the proof.
\end{proof}

\begin{cor}\label{gleichesing}
Let $A$ be a $g$-dimensional abelian variety over an algebraically closed field of characteristic $p=2$ and denote by $2^r= |\textnormal{ker}(2_A)(k)|$ the number of $2$-torsion points of $A$. We define  $B=E_1 \times \ldots \times E_g$ as the product of $r$ ordinary and $g-r$ supersingular elliptic curves. If $r=g$ or $r=g-1$ holds, then the singularities of the Kummer varieties of $A$ and $B$ are formal isomorphic.
\end{cor}


Hence, the description of the singularities of Kummer varieties arising from products of elliptic curves includes the case of arbitrary Kummer varieties with $2^g$ (ordinary case) or $2^{g-1}$ singular points.

\section{Open questions}

In this last section we discuss some questions that arise naturally through the description of the singularities and their open neighbourhood.

\medskip

\noindent
\emph{Embedding dimension}. For Kummer varieties $X$ arising from products of elliptic curves or ordinary abelian varieties, the embedding dimension of the singularities is always $2^g -1$, where $g$ denotes the dimension of $X$.  Now it is natural to ask whether this result still holds for arbitrary abelian varieties.
\begin{question}
Given a Kummer variety, is the embedding dimension of the singularities always $2^g-1$?
\end{question}

\medskip

\noindent
\emph{Rationality problem}. In dimension $g \leq 1$, every Kummer variety is isomorphic to $\mathbb{P}^0_k$ or $\mathbb{P}^1_k$, hence rational. The Kummer surface of a supersingular abelian variety is rational as well. This was proved by Shioda \cite{shioda} in the case of a product of supersingular elliptic curves and generalized by Katsura  \cite{katsura}. We call an abelian variety $A$ and its Kummer variety $A/\iota$ \emph{superspecial} if $A$ is isomorphic to a product of supersingular elliptic curves.

As soon as we have an open affine subscheme of the Kummer variety $X$, we can examine its function field and ask again whether the variety is rational. Shioda gives an explicit computation of the function field of the superspecial Kummer surface and gets $K=k(V_1, V_2, T)= k(r,s)$, where
\begin{align*}
 V_1 = \frac{r^2+s}{\omega r^2 s + r + \omega^2 s^2}, \ \quad V_2= \frac{r^2+s}{\omega^2 r^2 s + r + \omega s^2}, \ \quad T = V_1^2 V_2^2 r \left(V_1^{-1}+ \omega s\right)
\end{align*}
and $\omega$ denotes a primitive third root of unity.

For the superspecial Kummer threefold, we can argue in a similar way. It has  function field $K=k(V_1, V_2, V_3, T_{\{1,2\}}, T_{\{1,3\}}, T_{\{2,3\}}, T_{\{1,2,3\}})$ with the relations from Proposition~\ref{endlerzkalg} between the generators.  As we have
\begin{align*}
 T_{\{2,3 \}}= \frac{V_2^2 T_{\{1,3\}} + V_3^2 T_{\{1,2\}}+V_1^2 V_2^2 V_3^2}{V_1^2}, \quad \quad T_{\{1, 2, 3\}} = \frac{T_{\{1,2\}} T_{\{1,3\}}+ V_1 V_2^2 V_3^2}{V_1^2},
\end{align*}
these indeterminates can be omitted. Now for $i=2,3$ we can consider the two subfields $K_{i}=k(V_1, V_i, T_{\{1,i \}}) = k(r_i, s_i)$ of $K$ which are rational function fields, i.e. the indeterminates $V_1$, $V_2$, $V_3$, $T_{\{1,2\}}$, $T_{\{1,3\}}$ can be interpreted as rational functions in $r_2, s_2$ or $r_3, s_3$. As $V_1$ is contained in both $K_2$ and $K_3$, one gets the relation
\begin{align*}
 \frac{r_2^2+s_2}{\omega r_2^2 s_2 + r_2 + \omega^2 s_2^2} =  \frac{r_3^2+s_3}{\omega r_3^2 s_3 + r_3 + \omega^2 s_3^2} \quad  \left( = V_1 \right)
\end{align*}
in $K=k(r_2, s_2, r_3, s_3)$. It is unclear, whether the field $K$ is of the form $k(a,b,c)$.
\begin{question}
Is the superspecial Kummer variety (uni-)rational for some $g>2$?
\end{question}

\medskip

\noindent
\emph{Products of Artin--Schreier curves}. Instead of considering products of elliptic curves one can also study products of Artin--Schreier curves which are given by affine equations of the form $y^p-x^{p-1} y -(x+\alpha_2 x^2 + \ldots + \alpha_p x^p) $ over a ground field of characteristic $p$. Now $\mathbb{Z}/p\mathbb{Z}=\langle \sigma \rangle$ acts on the curve by $\sigma(x,y)=(x, y+x)$, having the fixed point $P=(0,0)$. Hence, the quotient of a product $C_1 \times \ldots \times C_n$ of these curves by the diagonal action has a singular point at the origin. As the group action is of the form described in Proposition~\ref{sr}, one can again reduce the problem of computing the quotient to the case of $\mathbb{Z}/p\mathbb{Z}$ acting on $\mathbb{A}^{2n}_k$ and use the result of Campbell and Hughes (Proposition~\ref{campbell_erz}).   Unfortunately, the relations between the generators are unknown for $p \neq 2$ and $n \geq 3$. In the case $n=2$ one gets the singularity from \cite{ito}, Proposition~2.2.
\begin{question}
What can be said about the singularities if $n \geq 3$? 
\end{question}

\end{document}